\documentclass{article}

\usepackage{amssymb,amsfonts,amsmath,amsthm}
\usepackage{hyperref}
\usepackage{parskip}
\usepackage{setspace}
\usepackage{fancyhdr}
\usepackage[dvipsnames]{xcolor}
\usepackage[numbers]{natbib}
\usepackage[bottom]{footmisc}
\usepackage[utf8]{inputenc}
\usepackage[english]{babel}
\usepackage{xcolor}
\usepackage{tree-dvips}

\usepackage{OldStandard}

\newcommand\myshade{85}
\colorlet{mylinkcolor}{blue}
\colorlet{mycitecolor}{red}
\colorlet{myurlcolor}{Aquamarine}
\hypersetup{
	linkcolor  = mylinkcolor,
	citecolor  = mycitecolor,
	urlcolor   = myurlcolor!\myshade!black,
	colorlinks = true,
}
\title{On certain extensions of Ramanujan's Master Theorem and their applications\footnote{Mathematics Subject Classifications: 	11S80, 11E45 , 35A22.\newline 
Keywords- Ramanujan's Master Theorem, gamma function, \textit{k}-gamma function, Mellin transform}}

\author{Omprakash Atale\footnote{Khandesh College Education Society's
Moolji Jaitha College Jalgaon-425001, Maharashtra, India. E-mail: atale.om@outlook.com} \,and Mahendra Shirude\footnote{Kavayitri Bahinabai Chaudhari North Maharashtra University Jalgaon-425001, Maharashtra, India. E-mail: dr.mahendrashirude@gmail.com}}

\date{May 23, 2021}

\begin{document}

\maketitle

\begin{abstract}
S. Ramanujan introduced a technique in 1913 for providing analytic expressions for certain Mellin-type integrals which is now known as Ramanujan's Master Theorem. This technique was communicated through his \textit{Quarterly Reports} and has a wide range of applicability in calculating the values of certain definite integrals. In this paper, we have presented some extensions of Ramanujan's Master Theorem that arise from the $k$-gamma function and the $p$-$k$ gamma function. Furthermore, we have established a Mellin-type double integral along with its corollaries. Some applications are established through the evaluation of a variety of definite integrals.      
\end{abstract}

\newtheorem{theorem}{Theorem}[section]
\newtheorem{corollary}{Corollary}[theorem]
\newtheorem{lemma}[theorem]{Lemma}
\theoremstyle{definition}
\newtheorem{definition}{Definition}[section]
\renewcommand\qedsymbol{$\blacksquare$}

\section{Introduction}
A stipulation in the scholarship obtained by Ramanujan in 1913 required him to submit quarterly reports detailing his research to the Board of Studies in Mathematics. There were three quarterly reports in total and they were never published. However, results from these quarterly reports were used by Hardy in his book on Ramanujan's work. One of the results from those quarterly reports is a theorem on Mellin transform today known as Ramanujan's Master Theorem [\cite{one}, p.298]\cite{two}.

\quad\,\, Ramanujan’s Master Theorem is a technique for evaluating the Mellin transform of analytic functions. It states that if \textit{f} has expansion of the form 
\begin{equation}
f(x)=\sum_{n=0}^{\infty}\phi(n)\frac{(-1)^{n}}{n !} x^{n},\label{opo}\tag{1.1}
\end{equation}
where $\phi(n)$ has a natural and continuous extension such that $\phi(0)\neq0$, then for $s>0$, we have
\begin{equation}
\int_{0}^{\infty} x^{s-1}\left(\sum_{n=0}^{\infty}\phi(n) \frac{(-1)^{n}}{n !} x^{n}\right) d x=\Gamma(s) \phi(-s)\label{opo}\tag{1.2}
\end{equation}
where \textit{s} is any positive integer. Ramanujan widely used the above result in computing certain definite integrals and infinite series. Ramanujan's proof of the above theorem was quite informal and did not align with formal techniques in mathematics. Ramanujan's proof does not say much about the conditions on $\phi(n)$ and the convergence of integrated and the integral itself.

\quad\,\, Hardy in \cite{fiv} provided a more detailed and intuitive proof of Ramanujan’s Master Theorem using Cauchy’s residue theorem and Mellin inversion formula which is next to be followed by an outline of his proof

\begin{theorem}\textit{[Hardy-Ramanujan Master Theorem]}
Let $\varphi(z)$ be an analytic (single-valued) function, defined on a half-plane $H(\delta)=\{z \in {C}: \Re (z) \geq-\delta\}$
for some $0<\delta<1 .$ Suppose that, for some $A<\pi, \phi$ satisfies the growth condition $|\phi(v+i w)|<C e^{P v+A|w|}$
for all $z=v+i w \in H(\delta)$, then for all $0<\Re s<\delta$
\begin{equation*}
\int_{0}^{\infty} x^{s-1}\left\{\phi(0)-x \phi(1)+x^{2} \phi(3) \ldots\right\} d x=\frac{\pi}{\sin s \pi} \phi(-s).\tag{1.3}\label{hardy}  \end{equation*}
\end{theorem}
\begin{proof}
Let $0<x<e^{-p}$ the growth condition shows that the series $\Phi(x)=\phi(0)-x \phi(1)+x^{2} \phi(3) \ldots$
converges. The residue theorem yields
\begin{equation}
    \Phi(x)=\frac{1}{2 \pi i} \int_{c-i \infty}^{c+i \infty} \frac{\pi}{\sin s \pi} \phi(-s) x^{-s} d s\tag{1.4}
\end{equation}
for any $0<c<\delta$. Observe that $\pi / \sin \pi s$ has poles at $s=-n$ for $n=0,1,2 \ldots$ with residue $(-1)^{n}$. The above integral converges absolutely and uniformly for $c \in(a, b)$ and $0<a<b<\delta$. The claim now follows from Theorem 1.1.
\end{proof}

\quad\,\, The substitution $\phi(u) \rightarrow \phi(u) / \mathrm{\Gamma}(u+1)$ establishes Ramanujan's Master Theorem in its original form (Eqn. (\ref{opo})). The condition $\delta<1$ in Theorem 1.2 ensures convergence of the integral \ref{hardy}. In Ref. \cite{six}, authors have shown that analytic continuation can be employed to Eqn. (\ref{hardy}) to validate it to a larger strip in which the integral converges.

\quad\,\, The goal of this paper is to explore some extensions of Ramanujan's Master Theorem for the zeta function and analogs of Ramanujan’s Master Theorem that arises from $k$-gamma function \cite{thr} and $p$-$k$ gamma function \cite{fou} which is also known as two-parameter gamma function.

\quad\,\, The paper is organized as follows. In sections 2 and 3, we stated the analogs of Ramanujan's Master Theorem for the $k$-gamma function and $p$-$k$ gamma function along with their examples. In section 4, we established certain extensions of Ramanujan's Master Theorem along with their corollaries and applications to calculate some definite integral. Next, in section 5, we have established a double integral identity of Ramanujan's Master Theorem along with its examples and corollaries.
\section{Ramanujan's Master Theorem for \textit{k} -gamma function}
In 2007, Rafael Diaz and Eddy Pariguan introduced a one-parameter deformation of the gamma function known as the $k-$gamma function. This function was motivated by the repeated appearance of the expression of the form $x(x+k)(x+2k)...(x+(n-1)k)$ in the variety of contexts in physics such as the combinatorics of creation and annihilation operators and perturbative computation of Feynman integrals. For areal number $k>0$, the $k-$gamma function is defined as
\begin{equation}
   \Gamma_{k}(s)=\int_{0}^{\infty} x^{s-1} e^{\frac{-x^{k}}{k}} d x\quad\,\,\Re (s)>0. \tag{2.1}
\end{equation}
The $k-$gamma function is related to the Euler's gamma function as follows:
\begin{equation}
    \Gamma_{k}(s)=k^{\frac{s}{k}-1} \Gamma\left(\frac{s}{k}\right).\tag{2.2}
\end{equation}
One can observe that in the limit as $k\to 1$, we get the Euler's gamma function.

\quad\,\, In this section, we are going to derive an $k$-gamma function analog of Ramanujan's Master Theorem (Theorem 2.1) with its further corollaries and application in computing certain definite integrals. 

\begin{theorem}
Suppose that, in some neighborhood of $x=0$,
\begin{equation}
   f_{k}(x)=\sum_{r=0}^{\infty}\phi(r)\frac{\left(-\frac{x^k}{k}\right)^{ r}}{r!} \tag{2.3}
\end{equation}
where $k>0$. Assuming that there exists a "natural" continuous extension of $\phi(r)$, such that $\phi(0) \neq$ 0, then
\begin{equation*}
\int_{0}^{\infty} x^{s-1}\left(\sum_{r=0}^{\infty}\phi(r)\frac{(-\frac{x^k}{k})^{ r}}{r!}\right) d x=\Gamma_{k}(s) \phi\left(\frac{-s}{k}\right)\tag{2.4}\label{jhd}
\end{equation*}
where \textit{s} is any positive integer.
\end{theorem}
\begin{proof}
The idea of proof for the above theorem is the same as what Ramanujan employed. We begin our proof with the definition of $\Gamma_{\mathrm{k}}$ function stated in [\cite{thr}, pg.4, Eqn. 5]
\begin{equation}
    \int_{0}^{\infty} x^{s-1} e^{\frac{-x^{k} m}{k}} d x=m^{\frac{-s}{k}} \Gamma_{k}(s).\tag{2.5}
\end{equation}

Replacing $m$ with $m^{n}$ with $m>0$ and $0 \leq n<\infty$ we obtain
\begin{equation}
   \int_{0}^{\infty} x^{s-1} e^{\frac{-x^{k} m^{n}}{k}} d x=m^{\frac{-ns}{k}} \Gamma_{k}(s) . \tag{2.6}
\end{equation}
Multiply both sides by $\frac{f^{(n)}(a)(h)^{n}}{n !}$ where $f$ shall be specified later, sum on $n$ and expand $e^{\frac{-x^{k} m^{n}}{k}}$ in it's
Maclurin series, we get
\begin{equation}
    \sum_{n=0}^{\infty} \frac{f^{(n)}(a)(h)^{n}}{n !} \int_{0}^{\infty} x^{s-1} \sum_{r=0}^{\infty} m^{n r}\frac{\left(-\frac{x^k}{k}\right)^{ r}}{r!} d x=\Gamma_{k}(s) \sum_{n=0}^{\infty} \frac{f^{(n)}(a)\left(h m^{\frac{-s}{k}}\right)^{n}}{n !},\tag{2.7}
\end{equation}
\begin{equation}
    \sum_{n=0}^{\infty} \frac{f^{(n)}(a)(h)^{n}}{n !} \int_{0}^{\infty} x^{s-1} \sum_{r=0}^{\infty} m^{n r}\frac{\left(-\frac{x^k}{k}\right)^{ r}}{r!}d x=\Gamma_{k}(s) f\left(h m^{\frac{-s}{k}}+a\right)\tag{2.8}
\end{equation}
where
\begin{equation}
    f\left(h m^{\frac{-s}{k}}+a\right)=\sum_{n=0}^{\infty} \frac{f^{(n)}(a)\left(h m^{\frac{-s}{k}}\right)^{n}}{n !}.\tag{2.9}
\end{equation}

Now, inverting the order of summation and using the Taylor's theorem, we obtain
\begin{equation}
    \int_{0}^{\infty} x^{s-1} \sum_{r=0}^{\infty} f\left(h m^{r}+a\right)\frac{\left(-\frac{x^k}{k}\right)^{r}}{r!}d x=\Gamma_{k}(s) f\left(h m^{\frac{-s}{k}}+a\right). \tag{2.10}
\end{equation}
Now define $f\left(h m^{r}+a\right)=\phi(r)$, where $a, h$ and $r$ are regarded as constants. Therefore, we obtain
\begin{equation*}
\int_{0}^{\infty} x^{s-1} \sum_{r=0}^{\infty} \phi(r) \frac{\left(-\frac{x^k}{k}\right)^{ r}}{r!} d x=\Gamma_{k}(s) \phi\left(\frac{-s}{k}\right).\tag{2.11}\label{comp}    
\end{equation*}
This completes our proof.
\end{proof}
\quad\,\, Making some simple substitutions in Eqn. (\ref{jhd}), it is possible to derive the $k-$gamma function analog of Hardy's version of Ramanujan's Master Theorem which we state in the form of the following corollary.
\begin{corollary}
For $k>0$, $\frac{s}{k}\notin\mathbb{Z}$ and $s$ subjected to conditions in Theorem 2.1, we have 
\begin{equation}
    \int_{0}^{\infty} x^{s-1} \sum_{r=0}^{\infty} \phi(r)(-x^k)^{r} d x=\frac{1}{k} \phi\left(\frac{-s}{k}\right) \frac{\pi}{\sin \frac{s \pi}{k}} .\tag{2.12}
\end{equation}
\end{corollary}
\begin{proof}
Substituting $\phi(r)$ with $\phi(r) \Gamma_{k}((r+1) k)$ in Eqn. (\ref{comp}), we get
\begin{equation*}
\int_{0}^{\infty} x^{s-1} \sum_{r=0}^{\infty} \phi(r)\Gamma_{k}((r+1) k) \frac{\left(-\frac{x^k}{k}\right)^{ r}}{r!} d x=\Gamma_{k}(s) \phi\left(\frac{-s}{k}\right) \Gamma_{k}(k-s).\label{tpt}\tag{2.13}    
\end{equation*}
We know from the definition of $\Gamma_{k}$ function [\cite{thr}, pg .4] that
\begin{equation*}
\Gamma_{k}(s)=k^{\frac{s}{k}-1} \Gamma\left(\frac{s}{k}\right)\quad\mathrm{and}\quad \Gamma_{k}(s k)=k^{s-1}(s-1) !,\label{sk}   \tag{2.14}
\end{equation*}
where $s \in \mathbb{N}$ and $k>0$.
The latter can be easily derived by replacing $s$ with $sk$ in the former one. Now, replacing s with $k-s$ in the former equation and $s$ with $r+1$ in the later equation in (\ref{sk})  gives

\begin{equation}
    \Gamma_{k}(k-s)=k^{-\frac{s}{k}} \Gamma\left(1-\frac{s}{k}\right)\quad\mathrm{and}\quad \Gamma_{k}((r+1) k)=k^{r} r !\tag{2.15}
\end{equation}
respectively. Substituting the above two values in Eqn.(\ref{tpt}), and writing $\Gamma_{k}(s)$ as $k^{\frac{s}{k}-1} \Gamma\left(\frac{s}{k}\right)$ in the above equation, we obtain
\begin{equation}
   \int_{0}^{\infty} x^{s-1} \sum_{r=0}^{\infty} \phi(r)(-x^k)^{ r} d x=\frac{1}{k} \phi\left(\frac{-s}{k}\right) \Gamma\left(\frac{s}{k}\right) \Gamma\left(1-\frac{s}{k}\right)\tag{2.17} 
\end{equation}
Now using the reflection formula for $\Gamma\left(\frac{s}{k}\right) \Gamma\left(1-\frac{s}{k}\right)$, the desire result readily follows.
\end{proof}

\quad\,\, Note that the above-mentioned corollary has singularities at integer values of $s$ and $k$ when $s|k$ due to the presence of sin term in the denominator of the right-hand side of the equations. So even if we have mentioned at the beginning of corollaries that the equations hold for $s, k>0$, it is to be noted that a pair of numbers for $s$ and $k$ should be chosen in a way such that ${\sin{\frac{s \pi}{2 k}}\neq{0}}$. Therefore we have included the condition $\frac{s}{k}\notin\mathbb{Z}$ in Corollary 2.1.1.
 
\quad\,\, Before we start exploring the applications of Theorem 2.1, which is $k$ -gamma function analog of Ramanujan's Master Theorem, note that besides from $k$-gamma function, Theorem 2.1 can also be derived from Eqn. (\ref{opo}) which is the original form of Ramanujan's Master Theorem. This can be done by following some minor substitutions. Therefore, replacing $x$ with $x^{k} / k$ in Eqn. (\ref{opo}) where $k>0$, we get
\begin{equation*}
\int_{0}^{\infty} \frac{x^{k s-k}}{k^{s-1}} f_{k}(x) x^{k-1} d x=\Gamma(s) \phi(-s)\label{kjk}\tag{2.6}    
\end{equation*}
where
$$
f_{k}(x)=\sum_{r=0}^{\infty} \phi(r) \frac{\left(-\frac{x^k}{k}\right)^{r}}{r !}.
$$
Now, replacing $s$ with $s / k$, we get
$$
\int_{0}^{\infty} \frac{x^{ s-k}}{k^{\frac{s}{k}-1}} f_{k}(x) x^{k-1} d x=\Gamma(\frac{s}{k}) \phi(\frac{-s}{k}).
$$
By further simplification, the desired conclusion readily follows. This is a comparatively easy way of deriving Eqn. (\ref{jhd}). However, we provided a detailed proof of it in the form of Theorem 2.1 to make a transparency of how Ramanujan would have derived his Master Theorem for the $k$ -gamma function instead of the original gamma function. Furthermore, Theorem 2.1 and its proof can be reduced to its original form (just as Ramanujan provided in his \textit{Quarterly Reports} [\cite{one}, pg .298,\cite{two}]) by substituting $k=1$.

\quad\,\, Now we state a few applications of Theorem 2.1.
\begin{corollary}
    If $\left|x^{k}\right|<k$, then for $m, s>0$, we have
\begin{equation}
\int_{0}^{\infty} x^{s-1}\left(1+\frac{x^{k}}{k}\right)^{-m} d x=\frac{k^{\frac{s}{k}-1} \Gamma\left(\frac{s}{k}\right) \Gamma\left(m-\frac{s}{k}\right)}{\Gamma(m)}.    
\end{equation}

\end{corollary}

\begin{proof}
Here, we have used the following binomial expansion for $m>0$
$$
(1+x)^{-m}=\sum_{r=0}^{\infty} \frac{\Gamma(r+m)}{\Gamma(m) r !}(-x)^{r}.
$$
Replace $x$ with $x^{k} / k$ and use Eqn. (\ref{jhd}) for $f_{k}(x)=\left(1+\frac{x^{k}}{k}\right)^{-m}$. This completes our proof.
\end{proof}
\begin{corollary}
If $s>0$, then
\begin{equation*}
\int_{0}^{\infty} x^{s-1} f^{r}\left(a-\frac{x^{k}}{k}\right) d x=k^{\frac{s}{k}-1} \Gamma\left(\frac{s}{k}\right) f^{\left(r-\frac{s}{k}\right)}(a)\label{tps}\tag{2.7}    
\end{equation*}
where $f^{r}(t)$ denotes the $r^{th}$ fractional derivative of f.
\end{corollary}
\begin{proof}
By Taylor's theorem [\cite{one}, pg .330], we have
\begin{equation*}
f^{r}(a-x)=\sum_{n=0}^{\infty} \frac{f^{(r+n)}(a)(-x)^{n}}{n !}.\label{tpe}\tag{2.8}    
\end{equation*}
Replacing $x$ with $x^{k} / k$ and using Eqn. (\ref{jhd}) gives
\begin{equation*}
\int_{0}^{\infty} x^{s-1} f^{r}\left(a-\frac{x^{k}}{k}\right) d x=\int_{0}^{\infty} x^{s-1} \sum_{n=0}^{\infty} \frac{f^{(r+n)}(a)(-x)^{k n}}{n ! k^{n}} d x\tag{2.9}    
\end{equation*}
$$=k^{\frac{s}{k}-1} \Gamma\left(\frac{s}{k}\right) f^{\left(r-\frac{s}{k}\right)}(a).
$$
This proves the theorem.
\end{proof}
\begin{corollary}
we have
\begin{equation*}
\int_{0}^{\infty} f^{r}\left(a-\frac{x^{2 k}}{k}\right) d x=\frac{k^{\frac{1}{2k}-1}}{2} \Gamma\left(\frac{1}{2 k}\right) f^{\left(r-\frac{1}{2k}\right)}(a).\tag{2.10}    
\end{equation*}
\end{corollary}
\begin{proof}
$\operatorname{Set} s=\frac{1}{2}$ and replace $x$ by $x^{2}$ in Eqn. (\ref{tps}).
\end{proof}
\section{Ramanujan’s Master Theorem for \textit{p-k} gamma function}

This short section is organized similarly to the previous section. We will begin by stating a few definitions of the $p$-$k$ gamma function which can be found in \cite{fou}. Next, we will establish a $p$-$k$ gamma function analog of Ramanujan’s Master Theorem in the form of Theorem 3.1 which will be derived classically from the original Ramanujan’s Master Theorem rather than from the integral representation of $p$-$k$ gamma function which we did in the case of $k$-gamma function for Theorem 2.1 in the previous section. Furthermore, we will state some examples and applications of our derived analog in the form of solutions of certain definite integrals. Since theorems and applications in this section are partially motivated by the previous section, for the sake of simplicity and saving our precious time, most of the proofs in this section will be short.

\quad\,\,For $p, k>0$, we define $p$-$k$ gamma function as
\begin{equation}
    { }_{p} \Gamma_{k}(s)=\int_{0}^{\infty} x^{s-1} e^{\frac{-x^{k}}{p}} d x\tag{3.1}
\end{equation}
where $\Re{(s)}>0$. The relation of $p$-$k$ gamma function with $k$ -gamma function and the original gamma function is given by
\begin{equation}
    { }_{p} \Gamma_{k}(s)=\left(\frac{p}{k}\right)^{\frac{s}{k}} \Gamma_{k}(s)=\frac{p^{\left(\frac{s}{k}\right)}}{k} \Gamma\left(\frac{s}{k}\right).\tag{3.2}
\end{equation}
Now, we are in a position to state the following theorem.
\begin{theorem}
If $_{p} f_{k}(x)$ has expansion of the form
\begin{equation}
    \sum_{r=0}^{\infty} \phi(r) \frac{\left(-\frac{x^k}{p}\right)^{r}}{r !}\tag{3.3}
\end{equation}
where p, $k>0$ and assuming that there exists a "natural" continuous extension of $\phi(r)$, such that $\phi(0) \neq 0$, then
\begin{equation*}
\int_{0}^{\infty} x^{s-1}\left(\sum_{r=0}^{\infty} \phi(r) \frac{\left(-\frac{x^k}{p}\right)^{r}}{r !}\right) d x={ }_{p} \Gamma_{k}(s) \phi\left(\frac{-s}{k}\right),\label{tpo}\tag{3.4}    
\end{equation*}
where s is any positive integer.
\end{theorem}
\begin{proof}
Theorem 3.1 can easily be derived by replacing $x$ with $x^{k} / p$ in Eqn. (\ref{opo}), then replacing $s$ with $s / k$ and using the same method that we used in sub-section $2.3 .$
\end{proof}

\textbf{Example 3.1.1.}  Consider the following expansion which can be easily derived by taking both sides logarithm of the Weierstrass product formula for the gamma function [\cite{six}, pg .8].
\begin{equation*}
\log \Gamma(1+x)=-\gamma x+\sum_{r=2}^{\infty} \zeta(r) \frac{(-x)^{r}}{k}\label{thrpnt}\tag{3.5},
\end{equation*}
which can be also written as
\begin{equation*}
\frac{\log \Gamma(1+x)+\gamma x}{x^{2}}=\sum_{r=0}^{\infty} \phi(r) \frac{(-x)^r}{r !}\label{thrpt}\tag{3.6}    
\end{equation*}
with
\begin{equation}
 \phi(r)=\frac{\Gamma(r+1) \zeta(r+2)}{r+2}.\tag{3.7}   
\end{equation}
Replacing $x$ with $x^{p} / k$ and using Theorem 3.1 for Eqn. (\ref{thrpt}) produces
\begin{equation*}
\int_{0}^{\infty} x^{s-1} \frac{\log \Gamma\left(1+\frac{x^{k}}{p}\right)+\gamma \frac{x^{k}}{p}}{\frac{x^{2 k}}{p^{2}}} d x={ }_{p} \Gamma_{k}(s)\frac{\Gamma\left(1-\frac{s}{k}\right) \zeta\left(2-\frac{s}{k}\right)}{2-\frac{s}{k}}\label{tpf}\tag{3.8}   
\end{equation*}
valid for $0<s<1$. For s $=\frac{1}{2}$ and $s=\frac{3}{4}$, Eqn. $(\ref{tpf})$ yields
\begin{equation*}
\int_{0}^{\infty} \frac{\log \Gamma\left(1+\frac{x^{k}}{p}\right)+\gamma \frac{x^{k}}{p}}{\frac{x^{\frac{1+4 k}{2}}}{p^{2}}} d x=\frac{p^{\left(\frac{1}{2 k}\right)}}{k} \frac{\Gamma\left(\frac{1}{2 k}\right) \Gamma\left(1-\frac{1}{2 k}\right) \zeta\left(2-\frac{1}{2 k}\right)}{2-\frac{1}{2 k}}\tag{3.9}    
\end{equation*}
and
\begin{equation*}
\int_{0}^{\infty} \frac{\log \Gamma\left(1+\frac{x^{k}}{p}\right)+\gamma \frac{x^{k}}{p}}{\frac{x^{\frac{1+8 k}{4}}}{p^{2}}} d x=\frac{p^{\left(\frac{3}{4 k}\right)}}{k} \frac{\Gamma\left(\frac{3}{4 k}\right) \Gamma\left(1-\frac{3}{4 k}\right) \zeta\left(2-\frac{3}{4 k}\right)}{2-\frac{3}{4 k}} .\tag{3.10} \end{equation*}
\textbf{Example 3.1.2.} Differentiate Eqn.(\ref{thrpnt}) to get
\begin{equation*}
\psi(x+1)=-\gamma +\sum_{r=1}^{\infty} (-1)^{r+1} \zeta(r+1)x^k \tag{3.7}.
\end{equation*}
This can be also written as 
\begin{equation*}
\frac{\psi(x+1)+\gamma}{x}=\sum_{r=o}^{\infty}\zeta(r+2)\Gamma(r+1)\frac{(-x)^r}{r!}.\tag{3.8}\end{equation*}
Replacing $x$ with $x^{k} / p$ and using Theorem 3.1 produces
\begin{equation*}
\int_{0}^{\infty} x^{s-1} \frac{\psi\left(1+\frac{x^{k}}{p}\right)+\gamma}{\frac{x^{k}}{p}}d x={ }_{p} \Gamma_{k}(s)\zeta\left({2-\frac{s}{k}}\right)\Gamma\left({1-\frac{s}{k}}\right)\label{tpn}\tag{3.9}   
\end{equation*}
valid for $0<s<1$. For s $=\frac{1}{2}$ and $s=\frac{3}{4}$, Eqn. (\ref{tpn}) yields
\begin{equation*}
\int_{0}^{\infty} \frac{\psi\left(1+\frac{x^{k}}{p}\right)+\gamma}{\frac{x^{k+\frac{1}{2}}}{p}}d x=\frac{p^{\left(\frac{1}{2 k}\right)}}{k}\Gamma\left(\frac{1}{2 k}\right)\zeta\left({2-\frac{1}{2k}}\right)\Gamma\left({1-\frac{1}{2k}}\right)\tag{3.10}   
\end{equation*}
and
\begin{equation*}
\int_{0}^{\infty} \frac{\psi\left(1+\frac{x^{k}}{p}\right)+\gamma}{\frac{x^{k+\frac{1}{4}}}{p}}d x=\frac{p^{\left(\frac{3}{4 k}\right)}}{k}\Gamma\left(\frac{3}{4 k}\right)\zeta\left({2-\frac{3}{4k}}\right)\Gamma\left({1-\frac{3}{4k}}\right).\tag{3.11}
\end{equation*}
\section{Extensions of Ramanujan’s Master Theorem and their application to zeta function}
In this section, we will establish certain new identities of Ramanujan’s Master Theorem that relate to zeta function $\zeta(s)$ and completed zeta function $\xi(s)$. To show the accuracy of our result, we have derived a well-known zeta function identity from our established theorem that relates the zeta function and gamma function.

\begin{theorem}
If $F(x)$ admits the expansions of the following form which must be convergent
$$
F(x)=\sum_{m=1}^{\infty} f(m x)
$$
with
$$
f(x)=\sum_{n=0}^{\infty}(-1)^{n} \frac{\phi(n)}{n !}(x)^{n}
$$
and $\phi(n)$ has a natural and continuous extension such that $\phi(0) \neq 0$, then for $\Re{(s)}>0$, we have
\begin{equation*}
\int_{0}^{\infty} x^{s-1} F(x) d x=\zeta(s) \Gamma(s) \phi(-s).\label{fpno}\tag{4.1}    
\end{equation*}
\end{theorem}
\begin{proof}
Replacing $x$ with $m x$ for $m>0$ in Eqn. (\ref{opo}), we get
\begin{equation*}
\int_{0}^{\infty} x^{s-1} m^{s} f(m x) d x=\Gamma(s) \phi(-s).\tag{4.2}   
\end{equation*}
Shifting $m^{s}$ to the right-hand side in the above equation and summing on $m$ for all natural numbers, we get
\begin{equation*}
\int_{0}^{\infty} x^{s-1} \sum_{m=1}^{\infty} f(m x) d x=\sum_{m=1}^{\infty} m^{-s} \Gamma(s) \phi(-s) .\tag{4.3}    \end{equation*}
From the definition of the zeta function, we know that
\begin{equation*}
\zeta(s) \stackrel{\mathrm{def}}{=} \sum_{m=1}^{\infty}{m}^{-s} .\tag{4.4} \end{equation*}
This completes our proof.
\end{proof}
\begin{corollary}
 If $\Re{(s)}>0$, then
\begin{equation*}
\int_{0}^{\infty} \frac{x^{s-1}}{e^{x}-1} d x=\zeta(s) \Gamma(s).\tag{4.5}    \end{equation*}
\end{corollary}
\begin{proof}
 For $a>0$, Let
$$
f(m x)=e^{-a m x}
$$
and expand $e^{-a m x}$ in its Maclaurin series, to get
$$
e^{-a m x}=\sum_{n=0}^{\infty}(-1)^{n} \frac{a^{n}}{n !}(m x)^{n}
$$
and
$$
\phi(s)=a^{s} .
$$
Substituting the values of $f(m x)$ and $\phi(s)$ in Eqn. (\ref{fpno}), we get
\begin{equation*}
\int_{0}^{\infty} x^{s-1} \sum_{m=1}^{\infty} e^{-a m x} d x=\zeta(s) \Gamma(s) a^{-s}\tag{4.6}    
\end{equation*}

\begin{equation*}
\int_{0}^{\infty} \frac{x^{s-1}}{e^{a x}-1} d x=\zeta(s) \Gamma(s) a^{-s}.\tag{4.7}    
\end{equation*}
The proof follows by substituting $a=1$.
\end{proof}
\begin{corollary}
If
\begin{equation*}
F_{k}(x)=\sum_{m=1}^{\infty} f_{k}(m x)\tag{4.8}   
\end{equation*}
and
\begin{equation*}
f_{k}(m x)=\sum_{n=0}^{\infty}(-1)^{n} \frac{\phi(n)}{n !}\left(\frac{m x^{k}}{k}\right)^{n},\tag{4.9}
\end{equation*}
then for $\Re{(s)}>0$, we have
\begin{equation*}
\int_{0}^{\infty} x^{s-1} F_{k}(x) d x=\zeta(\frac{s}{k}) \Gamma_{k}(s) \phi(\frac{-s}{k}).\tag{4.10}    
\end{equation*}
\end{corollary}
\begin{proof}
Replace $x$ with $x^{k} / k$ in Eqn (\ref{fpno}) to get
\begin{equation*}
\int_{0}^{\infty}\left(\frac{x^{k}}{k}\right)^{s-1} \sum_{m=1}^{\infty} f\left(\frac{m x^{k}}{k}\right) x^{k-1} d x=\zeta(s) \Gamma(s) \phi(-s).\tag{4.11}   
\end{equation*}
Define the notation
$$
f\left(\frac{m x^{k}}{k}\right)=f_{k}(m x).
$$
Now, replace $s$ with $s$⁄$k$ and use the same method of simplification that we have used in sub-section 2.3 to derive Eqn. (\ref{kjk}).
\end{proof}
\begin{corollary}
If $\Re{(s)}>0$, then
\begin{equation*}
\frac{1}{2} \int_{0}^{\infty} \sum_{n=1}^{\infty} \frac{x^{s-1}}{e^{\pi n^{2} x}-1} d x=\frac{\zeta(\frac{s}{2})\xi(s)}{s(s-1)},\label{frpttl}\tag{4.12}    
\end{equation*}
where
\begin{equation*}
\xi(s) \stackrel{\text { def }}{=} \pi^{-\frac{s}{2}} s(s-1) \zeta(s) \Gamma\left(\frac{s}{2}\right).\tag{4.13}    
\end{equation*}
\end{corollary}
\begin{proof}
Replace $s$ with $s/2$ and $x$ with $\pi n^{2} x$ in Eqn. (\ref{fpno}) and sum on $n$ from 1 to $\infty$ to obtain
\begin{equation*}
\int_{0}^{\infty} x^{\frac{s}{2}-1} \sum_{n=1}^{\infty} \sum_{m=1}^{\infty} f\left(m \pi n^{2} x\right) d x=\pi^{-\frac{s}{2}}n^{-s} \zeta(\frac{s}{2}) \Gamma\left(\frac{s}{2}\right) \phi\left(\frac{-s}{2}\right).\tag{4.14}    
\end{equation*}
Now, for $a>0$, define
$$
f(x)=e^{-a x}
$$
therefore,
\begin{equation*}
\int_{0}^{\infty} x^{s-1} \sum_{n=1}^{\infty} \sum_{m=1}^{\infty} e^{-a m \pi n^{2} x} d x=\pi^{-\frac{s}{2}} \zeta(\frac{s}{2}) \Gamma\left(\frac{s}{2}\right) a^{\frac{-s}{2}},\tag{4.15}     
\end{equation*}
\begin{equation*}
\int_{0}^{\infty} \sum_{n=1}^{\infty} \frac{x^{s-1}}{e^{a \pi n^{2} x}-1} d x=\pi^{-\frac{s}{2}} \zeta(\frac{s}{2}) \Gamma\left(\frac{s}{2}\right) a^{\frac{-s}{2}}.\tag{4.16}
\end{equation*}
Substituting $a=1$ and using the definition of $\xi$ function, we get
\begin{equation*}
\frac{1}{2} \int_{0}^{\infty} \sum_{n=1}^{\infty} \frac{x^{s-1}}{e^{\pi n^{2} x}-1} d x=\frac{\zeta(\frac{s}{2}) \xi(s)}{s(s-1)}.\tag{4.17}
\end{equation*}
This completes our proof of Eqn. (\ref{frpttl})
\end{proof}
\begin{corollary}
If $v+1>s$ then for any real number $c$, we have
\begin{equation*}
\int_{0}^{\infty} x^{s-v} F(x^c) d x=\frac{1}{c}\zeta(\frac{s-v+1}{c}) \Gamma(\frac{s-v+1}{c}) \phi(\frac{v-s-1}{c}).
\end{equation*}
\end{corollary}
\begin{proof}
Replace $x$ with $x^c$ and $s$ with $s-v+1/c$ in Eqn. (\ref{fpno}).
\end{proof}
\textbf{Example 4.1.1.} consider the following two expansion [\cite{nin}, Eqn. (2.6),(2.8)]
\begin{equation*}
\frac{\sinh \left(v \sinh ^{-1} x\right)}{\sqrt{1+x^{2}}}=\sum_{k=0}^{\infty} \frac{2^{k}\Gamma\left(\frac{1+k+v}{2}\right) \Gamma\left(\frac{1+k-v}{2}\right)}{\pi\csc \left(\frac{-\pi k}{2}\right) \csc \left(\frac{\pi v}{2}\right)} \frac{(x)^{k}}{k !}\label{fpet}\tag{4.18}   
\end{equation*}
and
\begin{equation*}
\frac{\cosh \left(v \sinh ^{-1} x\right)}{\sqrt{1+x^{2}}}=\sum_{k=0}^{\infty} \frac{2^{k}\Gamma\left(\frac{1+k+v}{2}\right) \Gamma\left(\frac{1+k-v}{2}\right)}{\pi\sec \left(\frac{\pi k}{2}\right) \sec \left(\frac{\pi v}{2}\right)} \frac{(x)^{k}}{k !}\label{fpntn}.\tag{4.19}
\end{equation*}
Applying Theorem 4.1 on Eqn. (\ref{fpet}) and (\ref{fpntn}) yields
\begin{equation*}
\int_{0}^{\infty} x^{s-1} \sum_{m=1}^{\infty} \frac{\sinh \left(v \sinh ^{-1} m x\right)}{\sqrt{1+m^{2} x^{2}}} d x   =\frac{\zeta(s) \Gamma(s)  \Gamma\left(\frac{1-s+v}{2}\right) \Gamma\left(\frac{1-s-v}{2}\right)}{2^{s} \pi\csc \left(\frac{\pi s}{2}\right) \csc \left(\frac{\pi v}{2}\right)}\tag{4.20}
\end{equation*}
 valid for  $-1<\Re{(s)}<1-|\Re{(v)}|$, and
\begin{equation*}
\int_{0}^{\infty} x^{s-1} \sum_{m=1}^{\infty} \frac{\cosh \left(v \sinh ^{-1} m x\right)}{\sqrt{1+m^{2} x^{2}}} d x=\frac{\zeta(s) \Gamma(s) \Gamma\left(\frac{1-s+v}{2}\right) \Gamma\left(\frac{1-s-v}{2}\right)}{2^{s} \pi\sec \left(\frac{\pi s}{2}\right) \sec \left(\frac{\pi v}{2}\right)}\tag{4.21}    
\end{equation*}
valid for $0< \Re{(s)}<1-|\Re{(v)}|$ respectively.\newline
\textbf A generalized inverse binomial expansion of $(1+\alpha x)^{-v}$ is given by the following expansion
\begin{equation*}
(1+\alpha x)^{-v}=\sum_{k=0}^{\infty} \alpha^{k} \frac{\Gamma(v+k)}{\Gamma(v)} \frac{(-x)^{k}}{k !}.\tag{4.22}
\end{equation*}
Theorem 4.1 for the above expansion yields
\begin{equation*}
\int_{0}^{\infty} x^{s-1} \sum_{m=1}^{\infty}(1+\alpha m x)^{-v} d x=\frac{\zeta(s) \Gamma(s) \Gamma(v-s)}{\alpha^{s} \Gamma(v)} .\tag{4.23}    
\end{equation*}
valid for $|\arg \alpha|<\pi$ and $0<$ $\Re{(s)}< \Re{(v)}$.\newline
\textbf{Example 4.1.3.} Consider the following expansion.
\begin{equation*}
\left(\frac{2}{1+\sqrt{1+4 x}}\right)^{\mu}=\mu \sum_{k=0}^{\infty} \frac{\Gamma(\mu+2 k)}{\Gamma(\mu+k+1)} \frac{(-x)^{k}}{k !}.\tag{4.24}
\end{equation*}
Substitute $x$ with $m x$, sum on $m$ from 1 to $\infty$ and apply Theorem 4.1 to obtain
\begin{equation*}
\int_{0}^{\infty} x^{s-1} \sum_{m=1}^{\infty}\left(\frac{2}{1+\sqrt{1+4 m x}}\right)^{\mu} d x=\frac{\mu \Gamma(s) \zeta(s) \Gamma(\mu-2 s)}{\Gamma(\mu-s+1)}\tag{4.25}
\end{equation*}
valid for $0<\Re{(s)}<\Re{\left(\frac{\mu}{2}\right)}$ \cite{nin},

\section{A double integral generalization }   
In this section, we have established a double integral identity for Ramanujan’s Master Theorem.
\begin{theorem}
If $0<s<n+1$ where $n, s$ are any non-negative integer, then
\begin{equation*}
\int_{0}^{\infty} \int_{0}^{\infty} x^{s-1} \sin ^{n} t\left(\sum_{k=0}^{\infty}\phi(k)\frac{(-1)^{k}}{k !}(x t)^{k}\right) d x d t=\Gamma(s) \phi(-s) \phi(n, s)\label{lst}\tag{5.1}
\end{equation*}
\textit{with}
\begin{equation*}
\phi(n, s)=\int_{0}^{\infty} \frac{\sin ^{n} t}{t^{s}} d t.\tag{5.2}    
\end{equation*}
\end{theorem}
\begin{proof}
From the definition of the gamma function, we know that for $x, p>0$
$$
\frac{1}{t^{s}}=\frac{1}{\Gamma(s)} \int_{0}^{\infty} e^{-t x} x^{s-1} d x.
$$
Therefore
$$
\phi(n, s)=\frac{1}{\Gamma(s)} \int_{0}^{\infty} \sin ^{n} t \int_{0}^{\infty} e^{-t x} x^{s-1} dxdt.
$$
Substitute $x$ with $m x$ where $m>0$, replace $m$ with $m^{k}$ for $0 \leq k<\infty$, multiply both sides by $\frac{f^{(k)}(a)(h)^{k}}{k !}$ and sum on $k$ to obtain
$$
\sum_{k=0}^{\infty} \frac{f^{(k)}(a)\left(h m^{-s}\right)^{k}}{k !} \Gamma(s) \phi(n, s)=\sum_{k=0}^{\infty} \frac{f^{(k)}(a)(h)^{k}}{k !} \int_{0}^{\infty} \sin ^{n} t  \int_{0}^{\infty}e^{-m^{k} t x} x^{s-1} dxdt.
$$
Now, expand $e^{-m^{k} t x}$ in its Maclaurin's series, invert the order of summation and integration, and apply Taylor’s theorem to deduce
$$
f\left(h m^{-s}+a\right) \Gamma(s) \phi(n, s)=\int_{0}^{\infty} \sin ^{n} t \int_{0}^{\infty} \sum_{n=0}^{\infty} \frac{f\left(h m^{n}+a\right)(-x t)^{n}}{n !} x^{s-1} dxdt.
$$
Let $f\left(h m^{r}+a\right)=\phi(r)$ where, $h$ and $a$ are constants greater than zero. This completes our proof.
\end{proof}

\quad\,\, In Theorem 5.1, one might think of the complications that will arise while calculating the value of $\phi(n, s).$ For $s=1,2$ the issue can be resolved by the following integral formula [\cite{sev}, pg. 232, Entry 30(i)]
$$
\int_{0}^{\infty} \frac{\sin ^{2 n+1} x}{x} d x=\int_{0}^{\infty} \frac{\sin ^{2 n+2} x}{x^{2}} d x=\frac{\sqrt{\pi} \Gamma\left(n+\frac{1}{2}\right)}{2 n !}.
$$
For $s>2$ we have the following theorem.
\begin{theorem}
[\cite{sev}, pg. 232, Entry 30({ii})]
\textit{(i) If $s>2$ and $n-s+1>0$, then for $\phi(n, s)$the following equality holds true.}
$$
(s-1)(s-2) \phi(n, s)=n(n-1) \phi(n-2, s-2)-n^{2} \phi(n, s-2)
$$
\textit{(ii) If $n$ is a non-negative integer, then}
$$
\phi(2 n+3,3)=\frac{\sqrt{\pi}\left(n+\frac{3}{2}\right) \Gamma\left(n+\frac{1}{2}\right)}{4(n+1) !}
$$
\textit{and}
$$
\phi(2 n+4,4)=\frac{\sqrt{\pi}(n+2) \Gamma\left(n+\frac{1}{2}\right)}{6(n+1)!}.
$$
\end{theorem}
Proof of the above theorem can be found in [\cite{sev}, pg. 232].\newline

\quad\,\,\textbf{Example 5.1.1.} Consider the following expansion of the trigonometric function [\cite{six}, pg. 7]\cite{eig}.
\begin{equation*}
\frac{\cos \left(v \tan ^{-1} \sqrt{x}\right)}{(1+x)^{v / 2}}=\sum_{k=0}^{\infty} \frac{\Gamma(v+2 k) \Gamma(k+1)}{\Gamma(v) \Gamma(2 k+1)} \frac{(-x)^{k}}{k !}.\tag{5.3}    
\end{equation*}
In order to apply Theorem 5.1, substitute $x$ with $xt$ for $0<m<\infty$, sum on $m$ and let
$$
\phi(k)=\frac{\Gamma(v+2 k) \Gamma(k+1)}{\Gamma(v) \Gamma(2 k+1)}.
$$
Therefore, for $0<s<1 / 2$ and $v>2 s$ we get
\begin{equation*}
\int_{0}^{\infty} \int_{0}^{\infty} x^{s-1} \frac{\sin ^{n} t \cos \left(v \tan ^{-1} \sqrt{x t}\right)}{(1+x t)^{v / 2}} dxdt=\frac{\Gamma(s) \Gamma(v-2 s) \Gamma(1-s) \phi(n, s)}{\Gamma(v) \Gamma(1-2 s)}.\tag{5.4}    
\end{equation*}
Similarly, the expansion
\begin{equation*}
\frac{\sin \left(2 v \tan ^{-1} \sqrt{x}\right)}{\sqrt{x}(1+x)^{v}}=\sum_{k=0}^{\infty} \frac{\Gamma(2 v+2 k+1) \Gamma(k+1)}{\Gamma(2 v) \Gamma(2 k+2)} \frac{(-x)^{k}}{k !}\tag{5.5}    
\end{equation*}
after the application of Theorem 5.1 for $0<s<1$ and $v>s+\frac{1}{2}$ yields
\begin{equation*}
\int_{0}^{\infty} \int_{0}^{\infty} x^{s-1} \frac{\sin ^{n} t \sin \left(2 v \tan ^{-1} \sqrt{x t}\right)}{\sqrt{x t}(1+x t)^{v}} dxdt=\frac{\Gamma(s) \Gamma(2 v-2 s+1) \Gamma(1-s) \phi(n, s)}{\Gamma(2 v) \Gamma(2-2 s)}.\tag{5.6}    
\end{equation*}\newline
\textbf{5.1.2.}  Application of Eqn. $(\ref{lst})$ on Eqn. $(\ref{fpet})$ and $(\ref{fpntn})$ produces
$$
\int_{0}^{\infty} \int_{0}^{\infty} x^{s-1} \frac{\sin ^{n} t \sinh \left(\nu \sinh ^{-1} x t\right)}{\sqrt{1+x^{2} t^{2}}} dxdt.$$
\begin{equation*}
=\frac{\zeta(s) \Gamma(s)  \Gamma\left(\frac{1-s+v}{2}\right) \Gamma\left(\frac{1-s-v}{2}\right) \phi(n, s)}{2^{s} \pi\csc \left(\frac{\pi s}{2}\right) \csc \left(\frac{\pi v}{2}\right)}\tag{5.7}    
\end{equation*}    
and
$$
{\int_{0}^{\infty} \int_{0}^{\infty} x^{s-1} \frac{\sin ^{n} t \cosh \left(v \sinh ^{-1} x t\right)}{\sqrt{1+x^{2} t^{2}}} dxdt}
$$
\begin{equation*}
=\frac{\zeta(s) \Gamma(s)  \Gamma\left(\frac{1-s+v}{2}\right) \Gamma\left(\frac{1-s-v}{2}\right) \phi(n, s)}{2^{s} \pi\sec \left(\frac{\pi s}{2}\right) \sec \left(\frac{\pi v}{2}\right)}\tag{5.8}    
\end{equation*}
respectively.\newline
\textbf{Corollary 5.1.1.}
(i) \textit{If $_{p} f_{k}(x, t)$ has expansion of the form}
$$
\sum_{r=0}^{\infty} \phi(r)\frac{\left(-\frac{x^kt}{p}\right)}{r!} 
$$
\textit{where $p, k>0$, then, for $2<{s}/{k}<n+1$,}
\begin{equation*}
\int_{0}^{\infty} \int_{0}^{\infty} x^{s-1} \sin ^{n} t\left( \sum_{r=0}^{\infty} \phi(r)\frac{\left(-\frac{x^kt}{p}\right)}{r!} \right) d x d t={ }_{p} \Gamma_{k}(s) \phi\left(-\frac{s}{k}\right) \phi\left(n, \frac{s}{k}\right).\label{fivepn}\tag{5.9}    
\end{equation*}
 
\begin{proof}
 For the proof of Eqn. (\ref{fivepn}), replace $x$ with $x^k$⁄$p$ in Eqn. (\ref{lst}), replace $s$ with $s$⁄$k$ and follow the steps that we have used to derive Eqn. (\ref{kjk}) in section 2. 
\end{proof}
\textbf{Corollary 5.1.2.}
\textit{If $0<s<n+1$, then}
$$
\int_{0}^{\infty} \int_{0}^{\infty} x^{s-1} \sin ^{n} t f^{r}(a-x t) d x d t=\Gamma(s) f^{(r-s)}(a) \phi(n, s)
$$
\begin{proof}
Substitute $x$ with $x t$ in Eqn. (\ref{tpe}) to get
$$
f^{r}(a-x t)=\sum_{n=0}^{\infty} \frac{f^{(r+n)}(a)(-x t)^{n}}{n !}
$$
and
$$
\phi(s)=f^{(r+s)}(a)
$$
Now, apply Theorem 5.1 to get the desired result.
\end{proof}
\textbf{Corollary 5.1.3.}\textit{ If $0<s<n+1$ where $n, s$ are any non-negative integer such that $s>2$ then}\newline
\textit{(i)}
\begin{equation*}
\int_{0}^{\infty} \int_{0}^{\infty} \sin ^{n} t \frac{x^{s-1}}{e^{x t}-1} d x d t=\Gamma(s) \zeta(s) \phi(n, s).\tag{5.11}    
\end{equation*}
\textit{(ii)}
\begin{equation*}
\int_{0}^{\infty} \int_{0}^{\infty} \sin ^{2 n+3} t \frac{x^{2}}{e^{x t}-1} d x d t=\frac{\Gamma(s) \zeta(s) \sqrt{\pi}\left(n+\frac{3}{2}\right) \Gamma\left(n+\frac{1}{2}\right)}{4(n+1) !}.\tag{5.12}    
\end{equation*}
\textit{(iii)}
\begin{equation*}
\int_{0}^{\infty} \int_{0}^{\infty} \sin ^{2 n+4} t \frac{x^{3}}{e^{x t}-1} d x d t=\frac{\Gamma(s) \zeta(s) \sqrt{\pi}(n+2) \Gamma\left(n+\frac{1}{2}\right)}{6(n+1) !}.\tag{5.13}    
\end{equation*}
\begin{proof}
For (i) replace $x$ with $m x$ for $0<m<\infty$ and sum on $m$ from 1 to $\infty$ in Eqn. (\ref{lst}) to get    \begin{equation*}
\int_{0}^{\infty} \int_{0}^{\infty} x^{s-1} \sin ^{n} t \sum_{m=1}^{\infty} f(m x t) d x d t=\zeta(s) \Gamma(s) \phi(-s) \phi(n, s)\label{fivepfrt}\tag{5.14}   \end{equation*}                         
where
\begin{equation*}
f(m x, t)=\sum_{k=0}^{\infty}(-1)^{k} \frac{\phi(k)}{k !}(m x t)^{k}.\tag{5.15}    
\end{equation*}
Now, let $f(x, t)=e^{-a x t}$ which for $f(m x, t)$ yields $e^{-a m x t}$. Substituting the value of $f(m x, t)$ in Eqn. (\ref{fivepfrt}) with $\phi(s)=a^{s}$ we get
\begin{equation*}
\int_{0}^{\infty} \int_{0}^{\infty} x^{s-1} \sin ^{n} t \sum_{m=1}^{\infty} e^{-a m x t} d x d t=a^{-s} \zeta(s) \Gamma(s) \phi(n, s),    
\end{equation*}
\begin{equation*}
\int_{0}^{\infty} \int_{0}^{\infty} x^{s-1} \sin ^{n} t \frac{1}{e^{a x t}-1} d x d t=a^{-s} \zeta(s) \Gamma(s) \phi(n, s).   
\end{equation*}
The desired result follows by letting $a=1$. Corollaries (ii) and (iii) are special cases of corollary (i) which can be calculated using Theorem 5.2.
\end{proof}
\begin{corollary}
For any non-negative integer n, we have\newline
(i)
\begin{equation*}
\int\limits_0^\infty  {\int\limits_0^\infty  {{x^2}{{\sin }^{2n + 3}}\left( {\sum\limits_{k = 0}^\infty  {\phi \left( k \right)\frac{{{{\left( { - 1} \right)}^k}}}{{k!}}{{\left( {xt} \right)}^k}} } \right)} } dx = \frac{{\sqrt \pi  \Gamma \left( s \right)\phi \left( { - s} \right)\left( {n + \frac{3}{2}} \right)\Gamma \left( {n + \frac{1}{2}} \right)}}{{4\left( {n + 1} \right)!}}.\tag{5.16}    
\end{equation*}
(ii)
\begin{equation*}
\int\limits_0^\infty  {\int\limits_0^\infty  {{x^3}{{\sin }^{2n + 4}}\left( {\sum\limits_{k = 0}^\infty  {\phi \left( k \right)\frac{{{{\left( { - 1} \right)}^k}}}{{k!}}{{\left( {xt} \right)}^k}} } \right)} } dx = \frac{{\sqrt \pi  \Gamma \left( s \right)\phi \left( { - s} \right)\left( {n + 2} \right)\Gamma \left( {n + \frac{1}{2}} \right)}}{{6\left( {n + 1} \right)!}}.\tag{5.17}   
\end{equation*}
\end{corollary}
\begin{proof}
(i) follows by letting $s=3$ and replacing $n$ with $2n+3$ and (ii) by letting $s=4$ and then replacing $n$ with $2n+4$. Use Theorem 5.2 in both cases.
\end{proof}
Further applications and examples of the theorems that we have established in this paper will be presented elsewhere in a more detailed investigation. 

\section{Conclusion}
In this paper, we have derived some analogs of Ramanujan's Master Theorem. In particular, we have explored the analogs of Ramanujan's Master Theorem that arise from the $k-$gamma function and $p,k-$gamma function. Furthermore, we have explored some other analogs that involve the zeta function. And finally, in the last section, we have derived a double integral analog of Ramanujan's Master Theorem. Further work in this direction is in progress.

\section{Acknowledgement}
I would like to thank Dr. Chintaman Aage for his expert advice on the convergence of certain series and encouragement throughout this paper. Furthermore, I would like to thank Dr. Shashank Kanade for their valuable remarks and kind consideration.  

\end{document}